\documentclass[11pt,reqno,oneside]{amsart}

\usepackage{graphicx}
\usepackage{amssymb}
\usepackage{epstopdf}

\usepackage{amsmath,amsfonts,amsthm,mathrsfs,amssymb,cite}
\usepackage[usenames]{color}

\usepackage{graphics}
\usepackage{framed}

\usepackage{enumerate}
\usepackage{hyperref}
\usepackage{xcolor}
\hypersetup{
  colorlinks,
  linkcolor={blue!80!black},
  urlcolor={blue!80!black},
  citecolor={blue!80!black}
}
\usepackage[capitalize,noabbrev]{cleveref}

\numberwithin{equation}{section}

\newtheorem{theorem}{Theorem}[section]
\newtheorem{lemma}[theorem]{Lemma}
\newtheorem{proposition}[theorem]{Proposition}

\theoremstyle{definition}
\newtheorem{definition}[theorem]{Definition}

\theoremstyle{remark}
\newtheorem{remark}[theorem]{Remark}

\numberwithin{equation}{section}

\newcommand{\be}{\begin{equation}}
\newcommand{\ee}{\end{equation}}
\newcommand{\ba}{\begin{array}}
\newcommand{\ea}{\end{array}}

\newcommand{\bea}{\begin{eqnarray*}}
\newcommand{\eea}{\end{eqnarray*}}
\newcommand{\bean}{\begin{eqnarray}}
\newcommand{\eean}{\end{eqnarray}}

\def\R{\mathbb{R}}

\def\Cc{\mathcal{C}}

\def\n{\nabla}

\def\a{\alpha}

\def\d{\delta}

\def\p{\partial}
\def\O{\Omega}
\def\ds{\displaystyle}

\def\pointe{\mathbf{p}}

\title[Calder\'on's inverse inclusion problem with smooth shapes]{On Calder\'on's inverse inclusion problem with smooth shapes by a single partial boundary measurement}

\author{Hongyu Liu}
\address{Department of Mathematics, City University of Hong Kong, Hong Kong, China}
\email{hongyu.liuip@gmail.com, hongyliu@cityu.edu.hk}

\author{Chun-Hsiang Tsou}
\address{Department of Mathematics, Hong Kong Baptist University, Kowloon, Hong Kong, China}
\email{c-h\_tsou@hkbu.edu.hk}

\author{Wei Yang}
\address{School of Mathematics, Xiangtan University, Xiangtan, Hunan, China}
\email{yangwei@xtu.edu.cn}

\begin{document}

\begin{abstract}
We are concerned with the Calder\'on inverse inclusion problem, where one intends to recover the shape of an inhomogeneous conductive inclusion embedded in a homogeneous conductivity by the associated boundary measurements. We consider the highly challenging case with a single partial boundary measurement, which constitutes a long-standing open problem in the literature. It is shown in several existing works that corner singularities can help to resolve the uniqueness and stability issues for this inverse problem. In this paper, we show that the corner singularity can be relaxed to be a certain high-curvature condition and derive a novel local unique determination result. To our best knowledge, this is the first (local) uniqueness result in determining a conductive inclusions with general smooth shapes by a single (partial) boundary measurement.   

\medskip


\end{abstract}

\maketitle

\section{Introduction}\label{sect:1}

Let $\O$ and $D$ be bounded Lipschitz domains in $\mathbb{R}^n$, $n\geq 2$, such that $D\Subset\O$ and $\O\backslash\overline{D}$ is connected. Let $\eta\in L^\infty (D)$ be such that $\eta(\mathbf{x})>\eta_0\in\mathbb{R}_+$, $\mathbf{x}=(x_j)_{j=1}^n\in \O$, and $|\eta(\mathbf{x})-1|>0$, $\mathbf{x}\in\mathrm{neigh}(\p D)\cap D$, where $\mathrm{neigh}(\p D)$ denotes an open neighbourhood of $\p D$. Set
 \begin{equation}\label{form}
 \gamma(\mathbf{x})=1+(\eta-1)\chi_D(\mathbf{x}),\: \mathbf{x}\in \O.
 \end{equation} 
 Introduce 
 \[
 H_0^{-1/2}(\p \O):=\{f\in H^{-1/2}(\p\O); \int_{\p\O} f\, ds=0\},
 \]
and consider the following elliptic PDE problem for $u\in H^1(\O)$,
\begin{equation}\label{EQ_Calderon}
\begin{cases}
\ds{\mathrm{div}(\gamma\nabla u)=0}\qquad & \mbox{in}\ \ \O,\medskip\\
\ds{\frac{\p u}{\p\nu}=g\in H_0^{-1/2}(\p\O)} & \mbox{on}\ \ \p\O,
\end{cases}
\end{equation} 
where $\nu\in\mathbb{S}^{n-1}$ is the exterior unit normal vector to $\p\O$.
Associated with \eqref{EQ_Calderon}, we introduce the following Neumann-to-Dirichlet (DtN) map, $\Lambda_\gamma: H_0^{-1/2}(\p\O)\mapsto H^{1/2}(\p\O)$, 
\begin{equation}\label{eq:NtD}
\Lambda_\gamma(g)=u|_{\p\O},
\end{equation}
where $u\in H^1(\O)$ is the solution to \eqref{EQ_Calderon}. For a given $g\in H_0^{-1/2}(\p\O)$, we are concerned with the inverse inclusion problem of determining $D$ by the pair of boundary data $(g, \Lambda_\gamma(g))$, independent of $\eta$. That is, 
\begin{equation}\label{eq:ic1}
(g, \Lambda_\gamma(g))\rightarrow \p D,\quad \mbox{independent of $\eta$}. 
\end{equation}

The inverse inclusion problem constitutes a particular set of sub-problems for the celebrated Calder\'on's inverse problem. The Calder\'on problem is concerned with the recovery of the conductivity of a body by the associated boundary electric current and flux measurements. Mathematically, it can formulated as determining $\gamma$ by knowledge of $\Lambda_\gamma$ associated with the PDE system \eqref{EQ_Calderon}, where $\gamma\in L^\infty(\O)$ is not necessarily of the specific form \eqref{form}. Calder\'on's inverse problem arises from industrial applications of practical importance including electrical impedance tomography (EIT) in medical imaging and geophysical exploration. The problem was first posed and studied by A. P. Calder\'on \cite{CALDERON2006} and has a profound impact on the mathematical developments of inverse problems. On the other hand, in many physical and engineering applications such as detecting defects in a composite medium, one is concerned more about recovering the geometric shape of an inhomogeneous inclusion embedded in a homogeneous conductive body. This naturally leads to the inverse inclusion problem \eqref{eq:ic1}. In the practical point of view, many reconstruction methods have been developed in the literature for the inverse inclusion problem. They include the monotonicity-based method \cite{HU}, the factorization method \cite{GH}, the enclosure method \cite{I}, and the method of using the generalized polarization tensors deduced from the layer potential techniques \cite{ammari2004reconstruction}.

It is noted that knowing $\Lambda_\gamma$ is equivalent to knowing $(g, \Lambda_\gamma(g))\in H_0^{-1/2}(\p\O)$ $\times H^{1/2}(\p\O)$ for any $g\in H_0^{-1/2}(\p\O)$, which corresponds to infinitely many boundary measurements in the practical scenario. In order to recover a generic $\gamma(x)$, it is necessary to make use of infinitely 
many boundary measurements. In fact, one can see that the cardinality of the unknown $\gamma(x)$ is $n$, whereas the cardinality of the given data encoded into $\Lambda_\gamma$ is $2(n-1)$. Here, by cardinality, we mean the number of independent variables in a given quantity. In such a case, one has $2(n-1)\geq n$ for $n\geq 2$, which means that the inverse problem in recovering a generic $\gamma(x)$ is at least formally determined (indeed, it is formally determined when $n=2$, and over-determined when $n\geq 3$). Nevertheless, for the inverse inclusion problem \eqref{eq:ic1}, the cardinality of the unknown $\p D$ is $n-1$, and for a fixed $g$, the cardinality of the boundary measurement data $(g, \Lambda_\gamma(g))$ is also $n-1$. That is, the inverse inclusion problem \eqref{eq:ic1} is formally determined with a single measurement. Hence, it is unobjectionable to expect that one can establish the unique identifiability result for \eqref{eq:ic1} in a generic scenario.  However, this problem constitutes a long-standing open problem in the literature. 

In \cite{beretta2019lipschitz,AMR,ammari2017identification,moi2,moi3}, the uniqueness and reconstruction issues for \eqref{eq:ic1} were studied under multiple measurements. There are also several works for \eqref{eq:ic1} by a single boundary measurement, provided that the inclusion $D$ is a-priori known to belong to certain specific geometrical classes. If $D$ is within the radial geometry, uniqueness and stability results were established in \cite{fabes1999inverse,moi1,Kang_ball_R3} by a single measurement. In the case that $D$ is a convex polygon in $\R^2$ or a convex polyhedron in $\R^3$, the uniqueness results are respectively derived in \cite{friedman1989uniqueness} and \cite{barcelo1994inverse}. The reconstruction of an insulating curvilinear polygonal inclusion was considered in \cite{CHL}. In a recent paper \cite{StabilityPolygon}, we proved a logarithmic type stability in determining polygonal inclusions. In those studies mentioned above by a single measurement, it is a technical requirement that the content of the inclusion has to be uniform; that is, the conductivity $\eta$ of the inclusion in \eqref{eq:ic1} is a positive constant. Moreover, in all of the aforementioned literature except \cite{StabilityPolygon}, full boundary measurements are required. Here, by full boundary measurement, we mean that the measurement dataset $(\psi, \Lambda_\gamma(\psi))$ is given over the whole boundary $\p\O$. For comparison, the following partial boundary measurement was used in \cite{StabilityPolygon}: $(\psi, \gamma\p_\nu u|_{\Gamma_0})$ with $\mathrm{supp}(\psi)\subset\Gamma_0$, where $\psi=u|_{\p\O}$ and $\Gamma_0\Subset\p\O$ is a proper subset. We would like to mention the partial-data Calder\'on problem constitutes another challenging topic in the field of inverse problems (cf. \cite{CSU,IUY}). Nevertheless, it is pointed out that a mild condition was imposed for the study in \cite{StabilityPolygon} which depends on the a-priori knowledge of the underlying inclusion as well as the corresponding boundary input. 

The mathematical argument in \cite{StabilityPolygon} is of a localized feature, which is based on carefully studying the singular behaviors (in the phase space) of the solution to the conductivity problem \eqref{EQ_Calderon} around a corner point on the polygonal inclusion. In this paper, we show that the corner singularity in \cite{StabilityPolygon} can be relaxed to be a certain high-curvature condition. Indeed, the corner singularity can be regarded as having an extrinsic curvature being infinity. Our argument in tackling the singular behaviors of the solution to \eqref{EQ_Calderon} around an admissible high-curvature point on the boundary of the conductive inclusion is mainly motivated by a recent article \cite{Liu_curve}. However, the study in \cite{Liu_curve} mainly deals with high-curvatures occurring on the support of a parameter $q$, which is the coefficient for the lower-order term of an elliptic partial differential operator, namely $-\Delta+q$. In the current study, the high-curvatures enter into the coefficient of the leading-order term, namely $\gamma$ associated with $\nabla(\gamma\nabla u)$. It is pointed out that in \cite{StabilityPolygon}, quantitative stability estimates are established in determining polygonal inclusions in two dimensions, whereas in this paper, we are mainly concerned with the qualitative unique identifiability issue in any dimension $n\geq 2$. Finally, we would also like to mention in passing some recent related works \cite{ACL,B1,B2,BLa,BLb,BL17,BL172,BLX,CX,CDL,CDL2,LX} on characterizing the geometric singularities in the coefficients of certain partial differential operators and their implications to the related inverse inclusion problems. 

The rest of the paper is organized as follows. In Section 2, we derive some auxiliary results. Section 3 is devoted to the main results for the inverse inclusion problem. In the Appendix, we present some discussion about a generic condition in the main theorem. 

\section{Some auxiliary results}

In this section, we derive several auxiliary results that shall be of critical importance for establishing the unique determination results in Section 3. By following the treatment in \cite{Liu_curve}, we first introduce an important geometric notion for our study. 

\begin{definition}\label{def:k}
Let $K,L,M,\d$ be positive constants and $D$ be a bounded domain in $\R^n$, $n\geq 2$. A point $\pointe\in \p D$ is said to be an admissible $K$-curvature point with parameters $L,M,\d$ if the following conditions are fulfilled.
\begin{enumerate}
\item Up to a rigid motion, the point $\pointe$ is the origin $\mathbf{x}=\mathbf{0}$ and $\mathbf{e}_n:=(0,\cdots,0,1)$ is the interior unit normal vector to $\p D$ at $\mathbf{0}$.
\item Set $b=\sqrt{M}/K$ and $h=1/K$. There is a $\mathcal{C}^{2,1}$ function $w: B(\mathbf{0},b)\rightarrow \R_+\cup \{0\}$ with $B(\mathbf{0},b)\subset \R^{n-1}$ such that if 
\begin{equation}
D_{b,h}=B(\mathbf{0},b)\times(-h,h)\cap D,
\end{equation}
then
\begin{equation}
D_{b,h}=\{\mathbf{x}\in \R^n; |\mathbf{x}'|<b, -h<x_n<h, w(\mathbf{x}')<x_n<h\},
\end{equation}
where $\mathbf{x}=(\mathbf{x}',x_n)$ with $\mathbf{x}'\in \R^{n-1}$. Here and also in what follows, $B(\mathbf{x}, b)$ signifies a ball centered at $\mathbf{x}$ and of radius $b$. 
\item The function $w$ in (2) satisfies
\begin{equation}
w(\mathbf{x}')=K|\mathbf{x}'|^2+\mathcal{O}(|\mathbf{x}'|^3),\qquad \mathbf{x}'\in B(\mathbf{0}',b).
\end{equation}
\item $M\geq 1$ and there are $0<K_-\leq K\leq K_+<\infty$ such that 
\[ K_-|\mathbf{x}'|^2\leq w(\mathbf{x}')\leq K_+|\mathbf{x}'|^2,\qquad |\mathbf{x}'|<b, \]
\[M^{-1}\leq \frac{K_\pm}{K}\leq M, \qquad K_+-K_-\leq LK^{1-\d}.\]
\item The intersection $V=\overline{D_{b,h}}\cap \R^{n-1}\times\{h\}$ is a Lipschitz domain.
\end{enumerate}
\end{definition}

In what follows, we suppose that $\p D$ possesses an admissible $K$-curvature point $\pointe$. We shall work within an Euclidean system of coordinates, which is transformed from the cardinal one after a rigid motion such that $\pointe$ is the origin. We assume that in this coordinate system, the boundary $\p D$ in a neighborhood of $\pointe$ can be represented by the equation $x_n=w(\mathbf{x}')$, where $w(\mathbf{x}')$ is the function in Definition~\ref{def:k} satisfying $w(\mathbf{0}')=0, \n_{\mathbf{x}'} w(\mathbf{0}')=\mathbf{0}'$. We denote by $U_{b,h}:=B(\mathbf{0}',b)\times (-h, h)$ the cylinder centered at $\pointe$, and by $D_{b,h}:=U_{b,h}\cap D$ the neighborhood of $\pointe$, and by $\p S :=\p D_{b,h}\setminus \p D$ the complementary part of its surface.

We next study the local behaviors of the solution $u$ to \eqref{EQ_Calderon}. We recall that $\eta$ in \eqref{form} is a positive constant and $\eta\neq 1$. Nevertheless, we shall remark the case that $\eta$ is a variable function in Section 3. It is straightforward to show that \eqref{EQ_Calderon} is equivalent to the following transmission problem with $u_i:=u|_{D}$ and $u_e:=u|_{\O\backslash\overline{D}}$: 
\begin{equation}\label{EQ2}
\begin{cases}
     \triangle u_e =0 & \text{in}\ \ \O\setminus\overline{D}, \\
     \triangle u_i =0 & \text{in}\ \ D,\\
     u_i=u_e & \text{on}\ \ \p D,\\
     \eta\p_\nu u_i=\p_\nu u_e & \text{on}\ \ \p D,\\
     \p_\nu u_e=g & \text{on}\ \ \p \O. 
\end{cases}  
\end{equation}
In principle, we shall show that $u_e$ in \eqref{EQ2} cannot be harmonically extended across the admissible $K$-curvature point $\pointe\in\p D$, provided $K$ is sufficiently large. It is easily seen that the harmonic extension across $\pointe$ is equivalent to the analytic extension of $u_e$ across $\pointe$. Then using this non-extension result together with an absurdity argument, we can establish the unique identifiability results for the inverse problem \eqref{eq:ic1}. To that end, throughout the rest of this section, we assume on the contrary that the function $u_e$ admits a harmonic extension into $D_{b, h}$, which is still denoted by $u_e$. That is, when $u_e|_{D_{b,h}}$ is involved in what follows, it is actually referred to the harmonic extension mentioned above. Under such a situation, we shall derive several key properties of the solution $u\in H^1(U_{b,h})$ locally around the point $\pointe\in\p D$, which shall be used for the absurdity argument in Section~\ref{sect:3} for the inverse inclusion problem. 

\begin{proposition}\label{TIntId}
Let $u_0\in H^1_{loc}(U_{b,h})$ be harmonic in $U_{b,h}$. Then,
\begin{equation}\label{IntId}
(\eta-1)\int_{D_{b,h}}\n u\cdot \n u_0 \, d\mathbf{x}=\int_{\p S}(\eta\p_\nu u_i-\p_\nu u_e)u_0-(u_i-u_e)\p_\nu u_0\, ds,
\end{equation}
where $u=u_i\chi_D+u_e\chi_{\O\backslash\overline{D}}\in H^1(U_{b,h})$ is the solution to (\ref{EQ2}). 
\end{proposition}
\begin{proof}
The integral identity \eqref{IntId} can be directly verified by using Green's formula and the transmission conditions across $\p D$ of $u$.
\end{proof}
\begin{proposition}\label{prop:regularity}
Under the assumption that $u_e$ can be harmonically extended into $D_{b,h}$, one has $u_i \in \Cc^{1,\a}(\overline{D_{b,h}})$ and $u_e\in \Cc^{1,\alpha}(\overline{U_{b,h}}\backslash D_{b,h})$ for some $0<\a \leq 1$ in two and three dimensions. 
\end{proposition}
\begin{proof}
 Clearly, from \eqref{EQ2}, one sees that $u_i\in H^1(D_{b, h})$ satisfies $\Delta u_i=0$ in $D_{b, h}$, and 
 $$
 u_i=u_e\quad \mbox{and}\quad \p_\nu u_i=\eta^{-1}\p_\nu u_e\quad\mbox{on}\ \ {\p D_{b, h}\cap \p D}. 
 $$   
 According to Definition~\ref{def:k}, we know that ${\p D_{b, h}\cap \p D}$ is $\mathcal{C}^{2,1}$. Hence, by the regularity estimate up the boundary for elliptic PDEs (cf. \cite[Theorem 4.18]{McL}), we have that $u_i\in H^3(D_{b, h})$. Finally, by the Sobolev embedding, we readily see $u_i \in \Cc^{1,\a}(\overline{D_{b,h}})$ for some $0<\a \leq 1$ in two and three dimensions. $u_i\in \Cc^{1,\alpha}(\overline{U_{b,h}}\backslash D_{b,h})$ is obvious since we assume that $u_e$ allows a harmonic extension into $D_{b,h}$. 
\end{proof}

\begin{remark}
The H\"older regularity of $u^i$ up to interface $\p\O$ around $\pointe$ is an important ingredient in our subsequent argument. In Proposition~\ref{prop:regularity}, we derive the $H^3$-smoothness, which yields the desired H\"older regularity by the Sobolev embedding in two and three dimensions. For general dimensions greater than 3, one can follow a completely similar argument to derive the $\mathcal{C}^{1,\alpha}$ regularity of $u^i$ in $\overline{D_{b,h}}$, but subject to requiring that ${\p D_{b, h}\cap \p D}$ is $\mathcal{C}^{r+1,1}$, where $r:=\lceil{\frac n 2-1+\alpha}\rceil$. That is, one needs to modify $\mathcal{C}^{2,1}$ in item (2) in Definition~\ref{def:k} to be $\mathcal{C}^{r+1,1}$ for the general dimension case in order to have Proposition~\ref{prop:regularity}. However, only two and three dimensions are physically meaningful, and hence we mainly focus on the two and three dimensional cases. Nevertheless, it is emphasized that all our subsequent arguments can be extended to higher dimensions after some necessary but slight modifications.  
\end{remark}

\begin{theorem}\label{estimationT}
Let $u_0\in H^1_{loc}(U_{b,h})$ be a harmonic function, known as the CGO (Complex Geometric Optics) solution, constructed in the from
\begin{equation}\label{CGO}
u_0=\exp(\mathbf{\xi}\cdot \mathbf{x}),
\end{equation}
 with the parameter 
 \begin{equation}\label{eq:rho}
 \xi=\mathrm{i}\tau \hat{\mathbf{v}}-\tau \mathbf{e}_n\in\mathbb{C}^n, \ \tau\in\mathbb{R}_+,
 \end{equation}
 where
 \begin{equation}\label{eq:pv1}
 \hat{\mathbf{v}}:=\begin{cases}
 \ds{\frac{\n u_i(\pointe)-(\n u_i(\pointe)\cdot \mathbf{e}_n)\mathbf{e}_n}{\big|\n u_i(\pointe)-(\n u_i(\pointe)\cdot \mathbf{e}_n)\mathbf{e}_n\big|}} \ \ \ &\mbox{if\ \ $\nabla u_i(\pointe)\nparallel \mathbf{e}_n$,}\medskip\\
 \mathbf{e}_1:=(1,0,\ldots,0)\ \ &\mbox{if\ \ $\nabla u_i(\pointe)\parallel \mathbf{e}_n$}. 
 \end{cases}
 \end{equation}
 Then it holds that
 \begin{equation}\label{estimation}
 \begin{split}
 &C_{n,\eta,\a}|\n u_i(\pointe)|\\
 \leq&  \Vert u_i \Vert_{\Cc^{1,\a}(\overline{D_{b,h}})}(1+(\tau h)^{(n-1)/2})e^{\tau(\frac{1}{4K}-h)}\\
 &+\Vert u_i \Vert_{\Cc^{1,\a}(\overline{D_{b,h}})}\left((\frac{K}{K_-})^{\frac{n-1}{2}}-(\frac{K}{K_+})^{\frac{n-1}{2}}\right)e^{\frac{\tau}{4K}}\\
&+\Vert u_i \Vert_{\Cc^{1,\a}(\overline{D_{b,h}})}(h+K^{-1}_-)^{\a/2}h^{(n+1+\a)/2}(K/K_-)^{(n-1)/2}\tau^{3/2}e^{\frac{\tau}{4K}}\\
&+(\Vert u_i \Vert_{\Cc^{1,\a}(\overline{D_{b,h}})}+\Vert u_e \Vert_{\Cc^{1,\a}(\overline{D_{b,h}})})h^{\a+(n-1)/2}(K/K_-)^{(n-1)/2}\\
&\quad \times (1+\tau h)\tau^{(n-1)/2}e^{\tau(\frac{1}{4K}-h)},
 \end{split}
 \end{equation}
where $C_{n,\eta,\a}$ is a positive constant depending on $n, \eta$ and $\a$. 
\end{theorem}
\begin{proof}
It is directly verified that $u_0$ constructed in \eqref{CGO}--\eqref{eq:pv1} is a harmonic function. Next, we apply the constructed $u_0$ to the integral identity \eqref{IntId}. 
By straightforward calculations, we split the integral at the left hand side of (\ref{IntId}) into the following identity,
\begin{eqnarray}\label{IntId2}
\n u_i(\pointe)\cdot \xi \int_{x_n>K|\mathbf{x}'|^2}e^{\xi\cdot \mathbf{x}} dx=\n u_i(\pointe)\cdot \xi \int_{x_n>\max(h, K|\mathbf{x}'|^2)}e^{\xi\cdot \mathbf{x}}\, d\sigma_\mathbf{x} \nonumber\\
+\n u_i(\pointe)\cdot \xi \left( \int_{K|\mathbf{x}'|^2<x_n<h}e^{\xi\cdot \mathbf{x}}\, d\sigma_{\mathbf{x}}-\int_{D_{b,h}} e^{\xi\cdot \mathbf{x}}\, d\sigma_\mathbf{x} \right) \nonumber\\
+\int_{D_{b,h}} \xi\cdot(\n u_i(\mathbf{x})-\n u_i(\pointe)) e^{\xi\cdot \mathbf{x}}\, d\sigma_\mathbf{x}\nonumber \\
+\frac{1}{\eta-1}\int_{\p S}(\eta\p_\nu u_i-\p_\nu u_e)e^{\xi\cdot \mathbf{x}}-(u_i-u_e)\p_\nu e^{\xi\cdot \mathbf{x}}\, ds.
\end{eqnarray}
For notational convenience, we rewrite the integral identity (\ref{IntId2}) in the following form,
\begin{equation}\label{eq:eee1}
\n u_i(\pointe)\cdot \xi I_0=\n u_i(\pointe)\cdot \xi I_1+\n u_i(\pointe)\cdot \xi I_2+ I_3+\frac{1}{\eta-1}I_4.
\end{equation}
where $I_j$, $j=1,\ldots,4$, are respectively defined as
\begin{equation}\label{eq:I}
\begin{split}
I_0=& \int_{x_n>K|\mathbf{x}'|^2}e^{\xi\cdot \mathbf{x}}\, d\sigma_\mathbf{x},\\
I_1=& \int_{x_n>\max(h,K|\mathbf{x}'|^2)}e^{\xi\cdot \mathbf{x}}\, d\sigma_\mathbf{x},\\
I_2=& \int_{K|\mathbf{x}'|^2<x_n<h}e^{\xi\cdot \mathbf{x}}\, d\sigma_\mathbf{x}-\int_{D_{b,h}} e^{\xi\cdot \mathbf{x}}\, d\sigma_\mathbf{x},\\
I_3=& \int_{D_{b,h}} \xi\cdot(\n u_i(\mathbf{x})-\n u_i(\pointe)) e^{\xi\cdot \mathbf{x}}\, d\sigma_\mathbf{x}, \\
I_4=& \int_{\p S}(\eta\p_\nu u_i-\p_\nu u_e)e^{\xi\cdot \mathbf{x}}-(u_i-u_e)\p_\nu e^{\xi\cdot \mathbf{x}}\, ds.
\end{split}
\end{equation}
Using Lemmas 2.17--2.20 in \cite{Liu_curve} we have the following estimates of the integrals $I_0,I_1,I_2,I_3,$:
\begin{equation}\label{eq:E1}
\begin{split}
I_0=& \frac{1}{-\xi_n}\left(\frac{\pi}{-\xi_n K}\right)^{(n-1)/2}\exp\left( -\frac{\xi'\cdot\xi'}{4\xi_n K} \right),\\
|I_1|\leq & C_n \frac{1+(\tau h)^{\frac{n-1}{2}}}{\tau^{\frac{n+1}{2}}K^{\frac{n-1}{2}}}e^{-\tau h},\\ 
|I_2|\leq & C_n \left(K_-^{-\frac{n-1}{2}}-K_+^{-\frac{n-1}{2}}\right)\tau^{-\frac{n+1}{2}},\\ 
|I_3|\leq & C_{n,\a}\Vert u_i \Vert_{\Cc^{1,\a}(\overline{D_{b,h}})}(h+K_-^{-1})^{\a/2}h^{(n+\a+1)/2}K_-^{-(n-1)/2},
\end{split}
\end{equation}
where $\xi=(\xi_j)_{j=1}^n$ and $\xi'=(\xi_j)_{j=1}^{n-1}$. 
By following a similar argument to the proof of Proposition 2.21 in \cite{Liu_curve}, the last integral $I_4$ can be estimated as follows,
\begin{equation}\label{eq:I4}
|I_4|\leq C_{n,\a}h^{\a+(n-1)/2}K_-^{-(n-1)/2}(1+\tau h)e^{-\tau h}(\eta\Vert u_i\Vert_{\Cc^{1,\a}(\overline{D_{b,h}})}+\Vert u_e\Vert_{\Cc^{1,\a}(\overline{D_{b,h}})}).
\end{equation}
Finally, by applying the estimates in \eqref{eq:E1} and \eqref{eq:I4} to \eqref{eq:eee1}, together with grouping similar terms, one can arrive at the estimate (\ref{estimation}).

The proof is complete. 
\end{proof}

\section{Unique determination results for Calder\'on's inverse inclusion problem}\label{sect:3}

In this section, we consider the inverse problem \eqref{eq:ic1}. In principle, we aim to establish a local unique determination result as follows. Let $(D, \eta)$ and $(\widetilde D, \widetilde \eta)$ be two conductive inclusions supported in $\O$ as described in \eqref{form}. Let $\Gamma_0\Subset\p\O$ be an open subset and let $g\in H_0^{-1/2}(\p\O)$ with $\mathrm{supp}(g)\subset\Gamma_0$. If
\begin{equation}\label{eq:u1}
\Lambda_{\gamma}(g)=\Lambda_{\widetilde{\gamma}}(g)\quad\mbox{on}\ \ \Gamma_0, 
\end{equation}
where $\gamma$ and $\widetilde\gamma$ are given in \eqref{form}, respectively associated with $\eta$ and $\widetilde\eta$, then in a certain generic scenario as shall be detailed in what follows, $D\Delta\widetilde D:=(D\backslash\overline{\widetilde{D}})\cup(\widetilde{D}\backslash\overline{D})$ cannot possess a $K$-curvature point as described in Definition~\ref{def:k} with $K$ sufficiently large, which lies on $\p \Sigma$. Here, $\Sigma$ is the connected component of $\O\backslash\overline{D\cup\widetilde{D}}$ that connects to $\p\O$. In the simpler geometric setup with $D$ and $\widetilde{D}$ both being convex, this means $D\Delta\widetilde{D}$ cannot possess a $K$-curvature point with $K$ sufficiently large. To establish the local uniqueness results, we shall make use of the contradiction argument by assuming that $D\Delta\widetilde D$ possesses an admissible $K$-curvature point. Let $\pointe$ signify the aforementioned admissible $K$-curvature point. Without loss of generality, we assume that $\pointe\in \p D\cap \p \Sigma$; see Fig.\ref{dispo_inverse} for a schematic illustration. Let $u$ and $\widetilde u$ be the electric potential fields in \eqref{EQ2} associated with $\gamma$ and $\widetilde\gamma$, respectively. By virtue of \eqref{eq:u1}, we readily see that 
\begin{equation}\label{eq:u2}
u=\widetilde u\quad\mbox{and}\quad \p_\nu u=\p_\nu \widetilde u\quad\mbox{on}\ \ \Gamma_0. 
\end{equation}
Noting that $\Delta u=\Delta \widetilde u=0$ in $\Sigma$, we thus have from the unique continuation principle that $u=\widetilde u$ in $\Sigma$. The contradiction shall be established from the quantitative properties of the solution $u\in H^1(U_{b,h})$ locally around the point $\pointe\in\p D$. It is recalled that according to our discussion below \eqref{EQ2}, $u=u_i\chi_{D_{b,h}}+u_e\chi_{U_{b,h}\backslash\overline{D_{b,h}}}$ in $U_{b,h}$. By taking $b$ and $h$ sufficiently small, we can assume that $U_{b,h}$ stays away from $\widetilde D$ with a positive distance.  Since $u_e=\widetilde{u}$ in ${U_{b,h}\backslash\overline{D_{b,h}}}$, and $\widetilde{u}$ is obviously analytic in $U_{b,h}$, we clearly have that $u_e$ can be harmonically extended into $D_{b,h}$, which is exactly $\widetilde u$. 

To establish the uniqueness result mentioned above, we start with introducing an admissibility condition on the input $g\in H_0^{-1/2}(\partial\Omega)$ in \eqref{EQ2} for our inverse problem study.  

\begin{figure}
\includegraphics[scale=0.6]{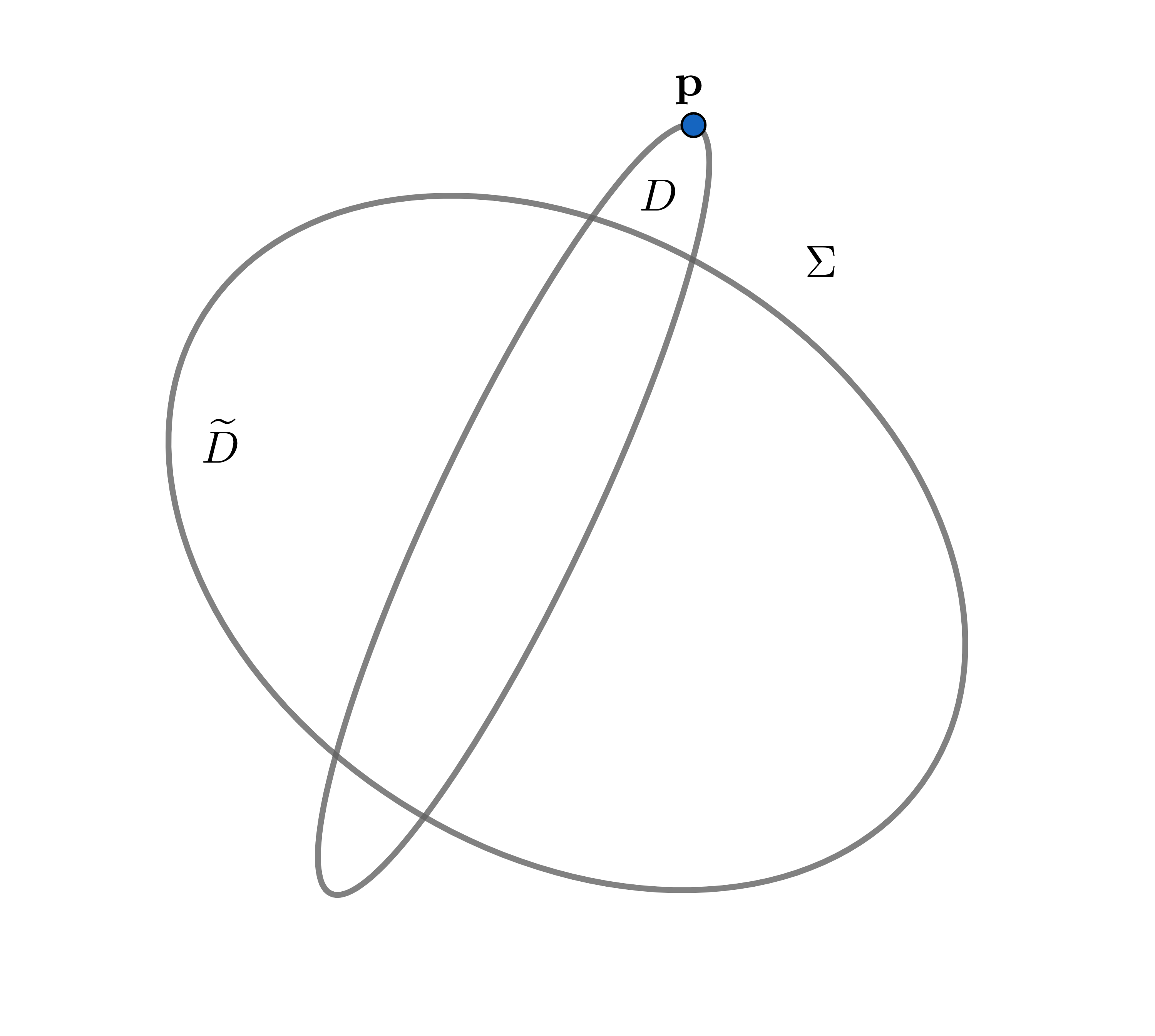} 
\caption{A $K$-curvature point and its neighborhood}
\label{dispo_inverse}
\end{figure}
\begin{proposition}\label{inf_grad}
Let $u$ be the solution to \eqref{EQ_Calderon}. Then there exists $g\in H_0^{-1/2}(\partial\Omega)$ such that for all $\mathbf{x}\in \O$,
\begin{equation}\label{eq:inf_grad}
\lim_{r\rightarrow +0}\frac{\int_{B_r(\mathbf{x})} |\n u(\mathbf{z})|d\mathbf{z}}{|B_r(\mathbf{x})|}\geq m_g,
\end{equation}
where $m_g>0$ is a positive constant independent of the conductivity function $\gamma$. Here and also in what follows, $B_r(\mathbf{x})$ signifies a ball centered at $\mathbf{x}$ with radius $r$.
\end{proposition}

\begin{remark}\label{rem:input}
In fact, the statement of Proposition \ref{inf_grad} in the 2-dimensional case is true if the input data $g$ has only one local maximum and only one local minimum on $\p \O$ (see e.g. \cite{gradu1,gradu2}). Through out the rest of the paper, we assume that the condition \eqref{eq:inf_grad} holds true in our study. In the context of the inverse inclusion problem, it means that the boundary input should be properly chosen. 
\end{remark}

\subsection{Local behaviours of the solution $u$ near the high curvature point $\pointe$}
In Theorem \ref{estimationT}, we obtain an estimate of $|\n u_i (\pointe)|$ in terms of the $K$-curvature parameters and the $\mathcal{C}^{1,\a}$-norm of $u_{i}$. In this subsection, we further refine the estimate \eqref{estimation}. Specifically, we shall need an estimates of $\Vert u_i \Vert_{\Cc^{1,\a}(\overline{D_{b,h}})}$ and $\Vert u_e \Vert_{\Cc^{1,\a}(\overline{D_{b,h}})}$ appeared in \eqref{estimation} in terms of the geometric parameters of the a-priori parameters. It is recalled that according to our discussion at the beginning of this section, $u_e=\widetilde u$ in $U_{b,h}$, with $\widetilde u$ satisfying \eqref{EQ2} associated with $(\widetilde D, \widetilde\eta)$ that arises from the contradiction argument in what follows. It is noted again that $\widetilde u$ is harmonic in $U_{b,h}$. 

By Proposition~\ref{prop:regularity}, we know that $u_i \in \Cc^{1,\a}(\overline{D_{b,h}})$ and $u_e\in \Cc^{1,\alpha}(\overline{U_{b,h}}\backslash D_{b,h})$ for some $0<\a\leq 1$. However, the H\"{o}lder norms of the solution $u$ on each side intricately depend on the geometric shape, and particularly depend on the local geometric parameter $K$ of the admissible $K$-curvature point $\mathbf{p}$. The following lemmas give an estimate of these H\"{o}lder norms in terms of the interface's local curvature \cite{DEF}.

\begin{lemma}\label{lemma_DEF}
Let $Q_R$ be a cube in $\R^n$ of side length $R$, centered at $\mathbf{p}\in\partial\Omega$. We denote the two sub-domains of $Q_R$ lying on the two sides of $\p D$ by $Q^+_R:=Q_R\cap \{x_n>w(\mathbf{x}')\}$ and $Q^-_R:=Q_R\cap \{x_n<w(\mathbf{x}')\}$, respectively. We consider the conductivity equation,
\begin{equation}\label{eq_DEF}
\mathrm{div} [(1+(\eta-1)\chi_\pm)\n u]=0 \ \ \text{in}\ \ Q_R,
\end{equation}
where $\chi_\pm(\mathbf{x})=1$ in $Q^+_R$ and $\chi_\pm(\mathbf{x})=0$ in $Q^-_R$. Then there exist positive constants $\a\in (0, 1),\mu,C$ independent of $u$ and $w$ such that
\begin{equation}\label{holder_estimate}
\Vert \n u\Vert_{\mathcal{C}^\a(Q^\pm_{R/4})}\leq C(1+\Vert \n w \Vert_{\mathcal{C}^1(Q_R)})^\mu \Vert \n u\Vert_{L^2(Q_R)}.
\end{equation}
\end{lemma}

\begin{remark}\label{rem:n1}
In the context of our study, we let $R_0>0$ be such that $U_{b,h} \Subset Q_{R_0}\Subset \O$, then Lemma~\ref{lemma_DEF} implies that there exist $C_{n,\eta,R_0}$ and $\mu$ which depend only on the a-priori data such that
\begin{equation}\label{holder_estimate_K}
\Vert \n u\Vert_{\mathcal{C}^\a(Q^\pm_{R_0/4})}\leq C_{n,\eta,R_0} K^\mu \Vert g \Vert_{H^{-1/2}(\p \O)}.
\end{equation}
\end{remark}

Clearly, \eqref{holder_estimate_K} gives the estimates of $\|\nabla u_i\|_{C^\alpha(\overline{D_{b,h}})}$ and $\|\nabla u_e\|_{C^\alpha(\overline{U_{b,h}}\backslash{D_{b,h}})}$. We next further derive the estimates of $\|u_i\|_{C(\overline{D_{b,h}})}$ and $\|u_e\|_{C(\overline{U_{b,h}}\backslash {D_{b,h}})}$ in terms of the local curvature of the admissible $K$-curvature point, which then yield the desired estimates of $\|u_i\|_{C^{1,\alpha}(\overline{D_{b,h}})}$ and $\|u_e\|_{C^{1,\alpha}(\overline{U_{b,h}}\backslash {D_{b,h}})}$.

\begin{lemma}\label{lem:holder}
Let $u=u_i\chi_D+u_e\chi_{\O\backslash\overline{D}}\in H^1(\Omega)$ be the solution to (\ref{EQ2}). Suppose that $D\Subset B(\mathbf{x}_0, r_0)\Subset\Omega$, where $B(\mathbf{x}_0, r_0)$ is a given ball. Furthermore, it is assumed that for any point $\mathbf{x}\in B(\mathbf{x}_0, r_0)\backslash {D}$, there exists a line segment within $\Omega$ which connects $\mathbf{x}$ to a point $\mathbf{x}'\in\Omega\backslash\overline{B(\mathbf{x}_0,r_0)}$. Then one has that 
\begin{equation}\label{eq:estn1}
\|u_i\|_{C^{1,\alpha}(\overline{D_{b,h}})}\leq C K^\mu\|g\|_{H^{-1/2}(\partial\Omega)},
\end{equation}
and
\begin{equation}\label{eq:estn2}
\|u_e\|_{C^{1,\alpha}(\overline{U_{b,h}}\backslash {D_{b,h}} )}\leq C K^\mu\|g\|_{H^{-1/2}(\partial\Omega)},
\end{equation}
where the positive constant $C$ depends on generic constant in the estimate \eqref{holder_estimate_K} as well as $\Omega$ and $B(\mathbf{x}_0, r_0)$, but independent of $K$. 
\end{lemma}

\begin{proof}
We first prove \eqref{eq:estn2}, and by virtue of \eqref{holder_estimate_K}, we see that
\begin{equation}\label{eq:estn25}
\|\nabla u_e\|_{C^\alpha(U_{b,h}\backslash\overline{D_{b,h}})}\leq C K^\mu\|g\|_{H^{-1/2}(\partial\Omega)}. 
\end{equation}
and hence it suffices for us to show that 
\begin{equation}\label{eq:estn3}
\|u_e\|_{C(U_{b,h}\backslash\overline{D_{b,h}})}\leq C K^\mu\|g\|_{H^{-1/2}(\partial\Omega)}. 
\end{equation}

By the standard elliptic PDE estimate, we know $u_e\in C(\Omega\backslash\overline{B(\mathbf{x}_0, r_0)})$, and for any $\mathbf{x}'\in\Omega\backslash\overline{B(\mathbf{x}_0, r_0)}$, 
\begin{equation}\label{eq:estn4}
|u_e(\mathbf{x}')|\leq C_1\|g\|_{H^{-1/2}(\partial\Omega)},
\end{equation}
where $C_1$ is a positive constant depending only on $\eta, \Omega$ and $B(\mathbf{x}_0, r_0)$, but independent of $K$. For any $\mathbf{x}\in \overline{U_{b, h}}\backslash{D_{b,h}}$, we let $\mathbf{x}'\in \Omega\backslash\overline{B(\mathbf{x}_0, r_0)}$ and ${l}(\mathbf{x},\mathbf{x}')$ be the line segment connecting $\mathbf{x}$ and $\mathbf{x}'$ such that $l(\mathbf{x},\mathbf{x}')\Subset\Omega$. Then by the intermediate value theorem, we have
\begin{equation}\label{eq:estn5}
u_e(\mathbf{x})-u_e(\mathbf{x}')=\nabla u_e(\xi)\cdot(\mathbf{x}-\mathbf{x}'),
\end{equation}
where $\xi\in l(\mathbf{x}, \mathbf{x}')$. By combining \eqref{eq:estn25} and \eqref{eq:estn4}, we readily have from \eqref{eq:estn5} that
\begin{equation}\label{eq:est6}
|u_e(\mathbf{x})|\leq C_2(1+K^\mu)\|g\|_{H^{-1/2}(\partial\Omega)}\leq 2C_2 K^\mu\|g\|_{H^{-1/2}(\partial\Omega)},
\end{equation}
where we have made use of the facts that $|\mathbf{x}-\mathbf{x}'|\leq \mathrm{diam}(\Omega)$, and without loss of generality that $K^\mu\geq 1$. \eqref{eq:est6} clearly implies \eqref{eq:estn3}, which in combination with \eqref{eq:estn25} immediately yields \eqref{eq:estn2}. 

\eqref{eq:estn1} can be proved by following a completely similar argument. Indeed, for any $\mathbf{x}\in D_{b, h}$, one can take $\mathbf{x}'\in \partial D\cap U_{b,h}$, and then it holds that
\begin{equation}\label{eq:estn7}
u_i(\mathbf{x})-u_e(\mathbf{x}')=u_i(\mathbf{x})-u_i(\mathbf{x}')=\nabla u_i(\xi)\cdot(\mathbf{x}-\mathbf{x}'),
\end{equation}
where $\xi\in l(\mathbf{x}, \mathbf{x}')\subset\overline{D_{b, h}}$. Finally, by combining Remark~\ref{rem:n1}, \eqref{eq:est6} and \eqref{eq:estn7}, one can show \eqref{eq:estn1}. 

The proof is complete.  

\end{proof}

\begin{remark}\label{rem:n2}
The geometric condition in Lemma~\ref{lem:holder}, namely for any point $\mathbf{x}\in B(\mathbf{x}_0, r_0)\backslash {D}$, there exists a line segment within $\Omega$ which connects $\mathbf{x}$ to a point $\mathbf{x}'\in\Omega\backslash\overline{B(\mathbf{x}_0,r_0)}$, can be easily fulfilled if $D$ is convex or star-shaped. 
\end{remark}

As mentioned earlier, we assume that $u_e$ can be harmonically extended into $D_{b,h}$, which is still denoted by $u_e$. That is, $u_e$ is harmonic in $U_{b,h}$ and hence is real analytic in $U_{b,h}$. Then for $b,h\in\mathbb{R}_+$ sufficiently small, we can have from \eqref{eq:estn2} that
\begin{equation}\label{eq:estn8}
\|u_e\|_{C^{1,\alpha}(\overline{U_{b,h}})}\leq C K^\mu\|g\|_{H^{-1/2}(\partial\Omega)},
\end{equation}
where $C$ depends on the same a-priori data as those in \eqref{eq:estn2}. Since throughout the paper, our argument is localized around an admission $K$-curvature point $\mathbf{p}$ (cf. Theorems~\ref{estimationT} and \ref{decayT}). Hence, in what follows, we shall always assume that \eqref{eq:estn8} holds true. By combining \eqref{estimation} of Theorem~\ref{estimationT} and Lemma~\ref{lem:holder}, one can derive the following theorem.

\begin{theorem}\label{decayT}
Let $u\in H^1(\Omega)$ be the solution to \eqref{EQ2} and $\mathbf{p}\in\partial D$ be an admissible $K$-curvature point. Suppose that \eqref{eq:estn1} and \eqref{eq:estn2} hold. We further suppose that the exponent $\mu$ in \eqref{eq:estn1} and \eqref{eq:estn2} (or equivalently in (\ref{holder_estimate_K})) satisfies
\begin{equation}\label{eq:cond1}
\mu <\frac{\min(1,\d)}{2}, 
\end{equation}
where $\delta$ is the a-priori parameter associated to $\mathbf{p}$ (cf. Definition~\ref{def:k}). 
Then it holds that
\begin{equation}\label{decay}
|\n u_i(\pointe)|\leq \mathcal{E}  \Vert g \Vert_{H^{-1/2}(\p \O)}(\ln K)^{(n+1)/2}K^{\mu-\min(\a,\d)/2},
\end{equation}
where $\mathcal{E}$ depends on the same a-priori data as those in \eqref{eq:estn1}--\eqref{eq:estn2} as well as $\alpha$ and $L, M$ in Definition~\ref{def:k} , but independent of $g$ and $K$. 
\end{theorem}
\begin{proof}
By plugging the estimates \eqref{eq:estn1} and \eqref{eq:estn8} into \eqref{estimation} of Theorem \ref{estimationT}, we can obtain
\begin{equation}\label{estimation_2}
\begin{split}
& C |\n u_i(\pointe)|\leq K^\mu \Vert g \Vert_{H^{-1/2}(\p \O)}(1+(\tau h)^{(n-1)/2})e^{\tau(\frac{1}{4K}-h)} \\
& +K^\mu \Vert g \Vert_{H^{-1/2}(\p \O)}\left((\frac{K}{K_-})^{\frac{n-1}{2}}-(\frac{K}{K_+})^{\frac{n-1}{2}}\right)e^{\frac{\tau}{4K}}\\
& +K^\mu \Vert g \Vert_{H^{-1/2}(\p \O)}(h+K^{-1}_-)^{\a/2}h^{(n+1+\a)/2}(K/K_-)^{(n-1)/2}\tau^{3/2}e^{\frac{\tau}{4K}} \\
& +K^\mu \Vert g \Vert_{H^{-1/2}(\p \O)}h^{\a+(n-1)/2}(K/K_-)^{(n-1)/2}\\
& \quad \times (1+\tau h)\tau^{(n-1)/2}e^{\tau(\frac{1}{4K}-h)}.
 \end{split}
 \end{equation}
 It is noted that in \eqref{estimation_2}, we have absorbed the generic constant involved into the right-hand side term, and it depends on the a-priori data as stated in the theorem, which should be clear in the context. 
 
Next, by following a similar argument to the proof of Proposition 2.22 in \cite{Liu_curve}, one can show that there exists $C_{n,L,M}>0$ such that 
 \begin{equation}
 \left|\left(\frac{K}{K_-}\right)^{\frac{n-1}{2}}-\left(\frac{K}{K_+}\right)^{\frac{n-1}{2}}\right|\leq C_{n,L,M}K^{-\d}.
 \end{equation}
 Using $h=1/K$ and $b=\sqrt{M}/K$ in the definition of the $K$-curvature point in Definition~\ref{def:k}, the estimate (\ref{estimation_2}) further yields
 \begin{eqnarray}\label{estimation_3}
& &C|\n u_i(\pointe)|\leq (1+(\tau /K)^{(n-1)/2})K^\mu e^{-\frac{3\tau}{4K}}\Vert g \Vert_{H^{-1/2}(\p \O)}\nonumber \\
& &+K^{\mu-\d}e^{\frac{\tau}{4K}}\Vert g \Vert_{H^{-1/2}(\p \O)}+K^{\mu-(n+2\a+1)/2} \tau^{3/2}e^{\frac{\tau}{4K}}\Vert g \Vert_{H^{-1/2}(\p \O)}\nonumber \\
& &+K^{\mu-\a-(n-1)/2} (1+\tau /K)\tau^{(n-1)/2}e^{-\frac{3\tau}{4K}}\Vert g \Vert_{H^{-1/2}(\p \O)}.
 \end{eqnarray}
Choosing $\tau=4K\ln K^\rho$ for some $\rho>0$ and dividing by $\Vert g \Vert_{H^{-1/2}(\p \O)}$, the left-hand side of (\ref{estimation_3}) can be estimated by
\begin{equation}\label{estimation_4}
(\ln K)^{(n-1)/2}K^{\mu-3\rho}+K^{\mu-\d+\rho}+(\ln K)^{3/2}K^{\mu+1-n/2-\a+\rho}+(\ln K)^{(n+1)/2}K^{\mu-\a-3\rho}.
\end{equation}
By setting $\rho=\min(\a,\d)/2$, each of the terms in (\ref{estimation_4}) can be estimated by
\begin{equation}
C(\ln K)^{(n+1)/2}K^{\mu-\min(\a,\d)/2},
\end{equation}
and thus the claim of this theorem follows.

The proof is complete. 
\end{proof}

\begin{remark}\label{rem:nn2}
In Appendix, we shall present two examples to numerically verify that the condition~\eqref{eq:cond1} can be fulfilled in generic scenarios of practical interest.
\end{remark}

\subsection{Local uniqueness result}

We are in a position to present the main local uniqueness result for the inverse inclusion problem \eqref{eq:ic1}. 

\begin{theorem}\label{mainT}
Let $(D, \eta)$ and $(\widetilde D, \widetilde\eta)$ be two conductive inclusions in $\Omega$, and $u, \widetilde u$ be the solutions to \eqref{EQ2} associated respectively to $(D, \eta)$ and $(D, \widetilde\eta)$. Suppose that  
\begin{enumerate}
\item $\partial D$ and $\partial\widetilde D$ are of class $C^{2, 1}$; 

\item $D$ and $\widetilde D$ satisfy the geometric condition in Lemma~\ref{lem:holder};

\item $G:=\O\setminus \overline{(D\cup \widetilde{D})}$ is connected;

\item the condition~\eqref{eq:cond1} is fulfilled for $u/\widetilde u$; 

\item Proposition \ref{inf_grad} holds for $u$ and $\widetilde{u}$. 
\end{enumerate}
Let $d_0\in\mathbb{R}_+$ and $\Gamma_0\subset\partial\Omega$. If $u=\widetilde u$ on $\Gamma_0$, then $D\Delta\widetilde{D}=(D\setminus \widetilde{D})\cup (\widetilde{D}\setminus D)$ cannot possess an admissible $K$-curvature point $\mathbf{p}$ such that 
\begin{equation}\label{eq:cc1}
\max\{\mathrm{dist}(\mathbf{p}, \partial D), \mathrm{dist}(\mathbf{p}, \partial\widetilde D)> d_0, 
\end{equation}
and $K\geq K_0$, where $K_0\in\mathbb{R}_+$ is sufficiently large and depends on the a-priori parameters of $\mathbf{p}$ in Definition~\ref{def:k} as well as $\Omega, d_0, g, \eta, \widetilde\eta$ and $B(\mathbf{x}_0, r_0)$ in Lemma~\ref{lem:holder}. 
\end{theorem}

\begin{proof}
By absurdity, we assume without loss of generality that there exists an admissible $K$-curvature point $\pointe\in \p D \cap \p G$, such that 
$B(\mathbf{p},d_0)\Subset\Omega\backslash\overline{\widetilde{D}}$. We next show that as $K\rightarrow+\infty$, $|\nabla u(\mathbf{p})|\rightarrow 0$, which yields a contraction the assumption (5) in the statement of the theorem.

First, let us consider the function $u-\widetilde{u}$. This function is harmonic in $G$. Moreover, $u-\widetilde{u}$ and $\p_\nu (u-\widetilde{u})$ vanish on $\Gamma_0$. It follows from the unique continuation property that $u=\widetilde{u}$ in $G$.

It is recalled that $u=u_i\chi_D+u_e\chi_{\Omega\backslash\overline{D}}$. Since $u=\widetilde u$ in $G$ and both $u$ and $\widetilde u$ are harmonic in $G$, we see that $u_e$ admits a harmonic extension $D\cap B(\mathbf{p}, d_0)$, which is actually $\widetilde u$. Hence, Theorem~\ref{decayT} applies and we immediately obtain from \eqref{decay}
 that
\begin{equation}\label{eq:nn2}
|\n u(\pointe)|\leq \mathcal{ET}  \Vert g \Vert_{H^{-1/2}(\p \O)}(\ln K)^{(n+1)/2}K^{\mu-\min(1,\d)/2}.
\end{equation}
Clearly, the right-hand side of the above estimate tends to zero as $K\rightarrow +\infty$. Therefore, we can choose $K_0$ such that when $K>K_0$, $|\n u(\pointe)| < m_g$ which contradicts to the assumption (5) stated in the theorem. 

The proof is complete. 
\end{proof}

\begin{remark}
It is remarked that the assumptions (2) and (3) in Theorem~\ref{mainT} can be fulfilled if $D$ and $\widetilde D$ are convex; see Remark~\ref{rem:n2}. Nevertheless, they might be more general than being convex in Theorem~\ref{mainT}. Moreover, since our argument is localized around the admission $K$-curvature point, one may consider a even more general geometric situation where there are multiple conductive inclusions. In such a case, one may also relax the requirement that $G:=\O\setminus \overline{(D\cup \widetilde{D})}$ is connected and replace $G$ to be the connected component of $\Omega\backslash\overline{D\cup\widetilde D}$ that connects to $\partial\Omega$. 
\end{remark}

\begin{remark}
As discussion in Remark~\ref{rem:input}, the condition~(5) in Theorem~\ref{mainT} can be fulfilled by choosing a suitable input $g$. As mentioned in Remark~\ref{rem:nn2}, we shall show in the Appendix that the condition~(4) can be fulfilled in generic scenarios. 
\end{remark}

\begin{remark}
Since our argument is localized around an admissible $K$-curvature point, it is sufficient for us to require that the conductivity parameter $\eta$ is constant in a small neighbourhood of the admissible $K$-curvature, and it can be an $L^\infty$ variable function in the rest part of the inclusion $D$. Our unique recovery result in Theorem~\ref{mainT} can be extended to such a case by straightforwardly modifying the relevant arguments. Finally, we would like to point out if sufficient a-priori information is available about the underlying inclusion, the local uniqueness result in Theorem~\ref{mainT} also implies a certain global uniqueness result, and we shall explore more along this direction in our future study. 
\end{remark}

\section*{Acknowledgment}
The work of H Liu was supported by the startup grant from City University of Hong Kong, Hong Kong RGC General Research Funds, 12302017, 12301218 and 12302919.

\appendix

\section*{Appendix.~Further remark on the condition (3.7)} 

\setcounter{section}{1}
\setcounter{equation}{0}

In Theorem~\ref{mainT}, we require that the condition~\eqref{eq:cond1} holds true, namely assumption (4). That is, the exponent $\mu$ in \eqref{eq:estn1} and \eqref{eq:estn2} (or equivalently in (\ref{holder_estimate_K}) in Remark~\ref{rem:n1}) is required to satisfy \eqref{eq:cond1}. The theoretical result in \cite{DEF} only shows that $\mu$ is a positive parameter, whereas we need a more precise upper bound of $\mu$ in order to establish the estimate of $|\n u_i(\pointe)|$ in Theorem \ref{decayT}. It is important to verify if this condition indeed holds in generic scenarios. However, a rigorous verification is fraught with significant challenges. In what follows, we present two general examples to numerically verify this condition indeed holds. Before that, it is noted that the requirement of the positive constant $\delta$ in Definition~\ref{def:k} is not restrictive, and hence in order to verify the condition~\eqref{eq:cond1}, it would be more appealing to see whether \eqref{eq:estn1} and \eqref{eq:estn2} can hold for $\mu<1/2$. 

Our numerical simulations below focus on the local behaviours of $|\n u|$ near the $K$-curvature point $\pointe$. For illustration, we only consider the two-dimensional case. We recall the configuration in Lemma \ref{lemma_DEF} and Remark~\ref{rem:n1}. Let $\pointe$ be a $K$-curvature point, and the interface in its neighbourhood can be represented by $\lbrace x_2=w(x_1) \rbrace$ for $\mathbf{x}=(x_1, x_2)$. Let $Q \Subset \R^2$ be a domain containing $\pointe$ such that $Q$ is divided into two non-empty sub-domains $Q^\pm$ by the interface. We consider the following conductivity equation, 
\begin{equation}\label{num_sim_eq}
\begin{cases}
\ds \mathrm{div}[(1+(\eta-1)\chi)\nabla u]=0\qquad & \mbox{in}\ \ Q,\medskip\\
\ds u=f\in H^{1/2}(\p Q) & \mbox{on}\ \ \p Q,
\end{cases}
\end{equation}
where $\chi$ is the characteristic function $\chi(\mathbf{x})=1$ if $x_2>w(x_1)$ and $\chi(\mathbf{x})=0$ if $x_2<w(x_1)$. The Dirichlet boundary condition $f$ is arbitrarily chosen. 

Our numerical experiments is to study the relationship between $\max_Q |\n u|$ and the curvature $K$. In order to do so, we choose two sets of interfaces, and we test different values of the curvature $K$ in each set of interfaces. The interfaces are precisely described by parametric curves where the point $\pointe$ is the point with the maximum curvature. In each set of interfaces, only the curvature $K$ at $\pointe$ is variable from case to case. To illustrate the relation \eqref{holder_estimate_K}, we draw the regression line of $\log(K)$ respect to $\log(\max_Q|\n u|)$. The slopes of the regression lines give the estimates of the values $\mu$ in the corresponding scenarios.

 The technical settings for numerical simulations are specified as follows:
 \begin{itemize}
\item We use Freefem++ \cite{freefem} as the FEM (finite element method) solver of the conductivity equation \eqref{num_sim_eq}.
\item $Q$ is the square of side $10$ centred at the origin in $\R^2$.
\item We choose a Dirichlet boundary condition on $\p Q$: $f=2x_1+3x_2$.
\item We choose two sets of interface functions.
\begin{enumerate}
\item Parabolic interface, $w(x_1)=Kx_1^2$.
\item Hyperbolic interface, $w(x_1)=A\sqrt{x_1^2+(\frac{A}{K})^2}$.
\end{enumerate}
\item The values of $K$ are taken as: $K=1,1.5,1.5^2,\cdots,1.5^9$.
\item We choose the conductivity $\eta=2$ in \eqref{num_sim_eq}.
\item For each value of $K$, we solve numerically the equation (\ref{num_sim_eq}) with the Dirichlet condition above.
\item We trace the regression line of $\log(K)$ respect to $\log(\max_Q|\n u|)$ for each set of the interface.
\end{itemize}
The numerical results are summarized in Figures~\ref{parabole} and \ref{hyperbole}). We can easily observe that the maximum of the gradient $\max_Q |\n u|$ indeed increase as $K$ grows. Moreover, we can estimate the value of $\mu$ respectively in the two cases by the regression lines.
\begin{figure}[!ht]
\includegraphics[scale=0.4]{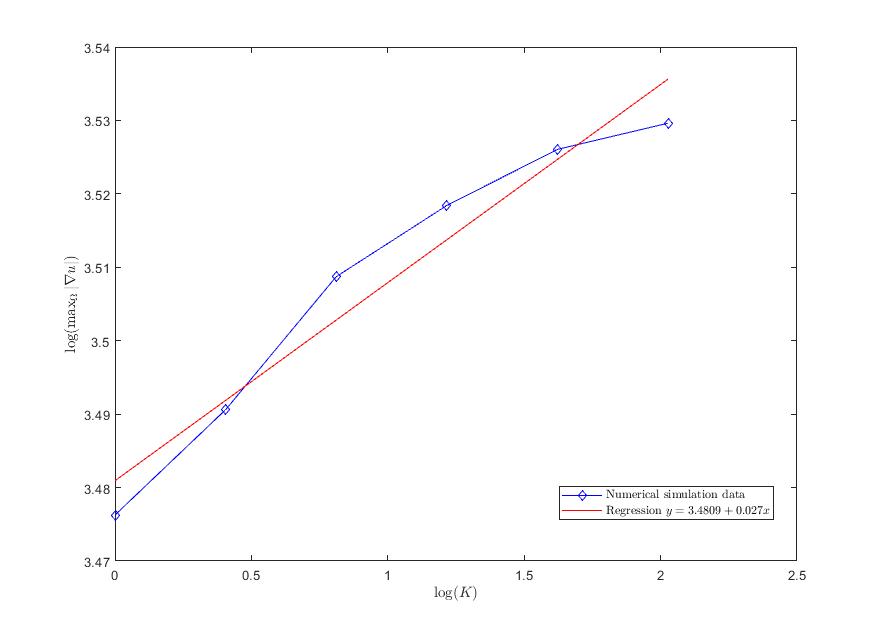} 
\caption{Parabolic interfaces}
\label{parabole}
\end{figure}
\begin{figure}[!ht]
\includegraphics[scale=0.4]{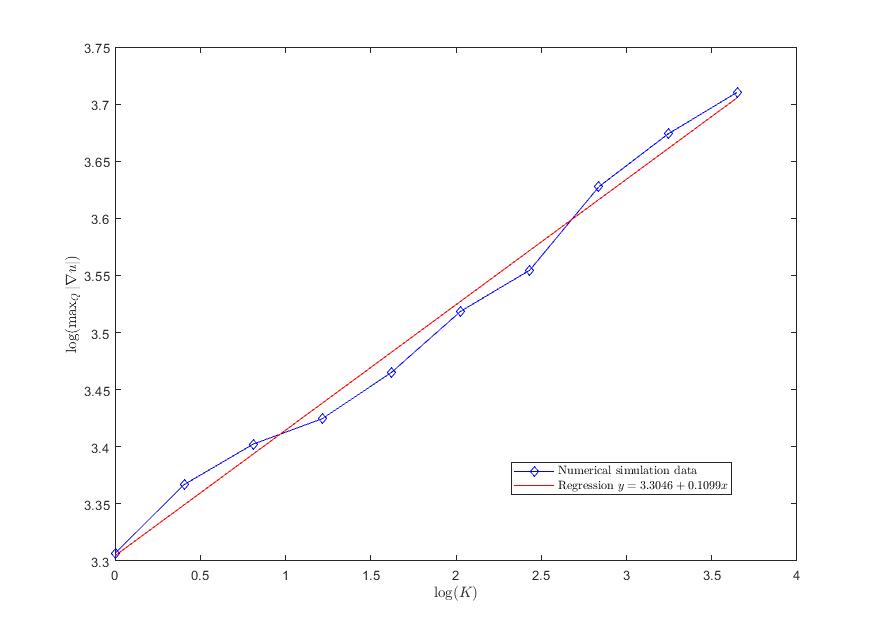} 
\caption{Hyperbolic interfaces}
\label{hyperbole}
\end{figure}
The regression lines indicate that in those two examples, the value of $\mu$ can be estimated as 
\begin{enumerate}
\item $\mu_{parabola}=0.027$;
\item $\mu_{hyperbola}=0.1099$.
\end{enumerate}
With those numerical results, we can conclude that the assumption (4) in Theorem \ref{mainT} is not void. In the case of parabolic interfaces, we have even observed a relatively weak dependence of $\n u$ on $K$. It is thus reasonable to make such an assumption to derive our local uniqueness result.

\end{document}